\documentclass{birkjour}
 \newtheorem{thm}{Theorem}[section]
 \newtheorem{cor}[thm]{Corollary}
 \newtheorem{lem}[thm]{Lemma}
 \newtheorem{prop}[thm]{Proposition}
 \theoremstyle{definition}
 \newtheorem{defn}[thm]{Definition}
 \theoremstyle{remark}
 
 \newtheorem*{ex}{Example}
 \numberwithin{equation}{section}
\usepackage[colorlinks=true,linkcolor=red,citecolor=blue]{hyperref}
\begin{document}
%-------------------------------------------------------------------------
% editorial commands: to be inserted by the editorial office
%
%\firstpage{1} \volume{228} \Copyrightyear{2004} \DOI{003-0001}
%\seriesextraline{This is the Concrete Title of this Book\br H.E. R and S.T.C. W, Eds.}
%\firstpage{1}

\title[ Sequentially Right-like properties on Banach spaces]
{Sequentially Right-like properties on Banach spaces}
%----------Author 1
\author {M. Alikhani.}
\address{Department of Mathematics, University of Isfahan}
\email{m2020alikhani@ yahoo.com}
%\thanks{This work was completed with the support of our \TeX-pert.}
%----------Author 2
%\author{A Second Author}
%\address{}
%\email{dont@know.who.knows}
\date{January 1, 2004}
\begin{abstract}
In this paper, we first study concept of $ p $-sequentially Right property, which is $ p$-version of
the sequentially Right property.\ Also,
we introduce a new class of subsets of Banach spaces which is called
$ p$-Right$ ^{\ast} $ set and obtain the relationship between p-Right subsets and p-Right$ ^{\ast} $
subsets of dual spaces.\ Furthermore, for $ 1\leq
p<q\leq\infty, $ we introduce the concepts of
properties $ (SR)_{p,q}$ and $ (SR^{\ast})_{p,q}$ in order to find a
condition which every Dunford-Pettis $ q $-convergent operator is Dunford-Pettis $p$-convergent.\ Finally, we apply these concepts and obtain
some characterizations of $ p $-Dunford-Pettis relatively compact property of Banach spaces and their dual spaces.
\end{abstract}
\maketitle
%\tableofcontents
\section{ Introduction }
Peralta et al. \cite{pvwy}, proved that for a given Banach space X there is a locally convex topology
on $ X, $ which is called the“ Right topology ”, such that a linear map $ T $ from $ X$ into a Banach
space $ Y $ is weakly compact if and only if it is Right-to-norm sequentially continuous.\
Later on Kacena \cite{ka} by introducing the notion of Right set in $ X^{\ast} $(dual of $ X$), showed that a Banach
space $ X $ has the sequentially Right property if and only if every Right subset of $ X^{\ast}, $ is relatively
weakly compact.\ A bounded subset $ K $ of $ X^{\ast} $ is a Right set, if every Right-null sequence $ (x_{n})_{n} $
in $ X $ converges uniformly to zero on $ K; $ that is,
$$\lim_{n\rightarrow \infty} \sup_{x^{\ast}\in K}\mid x^{\ast} (x_{n})\mid =0. $$
Recall that a bounded
subset $ K$ of Banach space $ X $ is a Dunford-Pettis set, if
every weakly null sequence $(x^{\ast}_{n})_n $ in $ X^{\ast}, $ converges uniformly to zero on the set $K$ \cite{An}.\ A Banach space $ X $ has the Dunford-Pettis relatively compact property (in short $ X$ has the $(DPrcP) $), if every Dunford-Pettis subset of $ X$ is relatively compact \cite{e1}.\\  For more information and examples of Banach spaces with Dunford-Pettis relatively compact property and sequentially Right property,
we refer to \cite{e1, ka, pvwy}.\\
Recently, Ghenciu \cite{g9} introduced the concepts of Dunford-Pettis $ p$-convergent operators,
$ p$-Dunford-Pettis relatively compact property (in short $ p$-$ (DPrcP)$), $ p $-Right sets and $ p $-sequentially
Right property (in short $ p$-$ (SR) $) on Banach spaces as follows:
\begin{itemize}
\item An operator $T: X\rightarrow Y $
is said to be Dunford-Pettis $ p$-convergent, if it takes Dunford-Pettis weakly $ p $-summable sequences to norm null sequences.\ The class of Dunford-Pettis $ p $-convergent operators from $ X $ into $ Y $ is denoted by $ DPC_{p}(X,Y). $\
\item A Banach space $ X$ has the $ p $-Dunford-Pettis relatively compact property ($ X $ has the $ p $-$ (DPrcP) $), if every Dunford-Pettis weakly p-summable sequence $ (x_{n})_{n} $ in $ X $ is norm null.\
\item A bounded subset $ K $ of $ X^{\ast} $ is called a $ p$-Right set, if every Dunford-Pettis weakly $ p$-summable sequence $ (x_{n})_{n} $ in $ X$ converges uniformly to zero on $ K,$ that is,
$$\displaystyle\lim_{n} \displaystyle\sup_{x^{\ast}\in K}\vert x^{\ast}( x_{n}) \vert=0.$$\
\item A Banach space $ X $ has the $ p $-sequentially Right property ($ X $ has the $ p$-$(SR) $), if every $ p $-Right set in $ X^{\ast}$ is relatively weakly compact.\
\end{itemize}

Motivated
by the above works, in Section 3, we introduce the concepts of
$ p $-Right$ ^{\ast} $ sets and $ p $-sequentially Right$ ^{\ast} $ property on Banach spaces.\  Then, we obtain the relationship between
$ p $-Right subsets and $ p $-Right$ ^{\ast} $ subsets of dual spaces.\ In addition, the stability of $ p $-sequentially Right property for some subspaces of bounded linear operators and projective tensor product between two Banach spaces are investigated.\\
In the Section 4, for $ 1\leq p<q\leq\infty $
inspired by the class $ \mathcal{P}_{p,q} $ in \cite{CAN}, for those Banach spaces in
which relatively $ p$-compact sets are relatively $ q$-compact, we
introduce the concepts of properties $ (SR)_{p,q}$ and $ (SR^{\ast})_{p,q}$ for those
Banach spaces in which Dunford-Pettis $ q $-convergent operators are Dunford-Pettis $ p $-convergent operators.\
Finally, by applying these
concepts, some characterizations for the $ p$-Dunford-Pettis relatively compact property of Banach spaces and
their dual spaces are investigated.\
Note that, the our results are motivated by results in \cite{bpo,g9, ka, pvwy,CAN}.\
\section{Definitions and Notions}
Throughout this paper $ 1\leq p < \infty ,$ $ 1\leq p < q\leq \infty,$ except for the cases where we consider other assumptions.\ Also, we suppose
$ X,Y $ and $ Z$ are arbitrary Banach spaces, $p^{\ast}$ is the H$\ddot{\mathrm{o}}$lder conjugate of $p;$ if $ p=1,~~ \ell_{p^{\ast}} $
plays the role of $ c_{0} .$\ The unit coordinate vector in $ \ell_{p} $ (resp.\ $ c_{0} $ or $\ell_{\infty} $) is denoted by $ e_{n}^{p} $ (resp.\ $ e_{n} $).\
The space $ X $ embeds in $ Y, $ if $ X $ is isomorphic to a closed subspace of $Y$ (in short we denote $ X\hookrightarrow Y $).\ We denote two isometrically isomorphic spaces $ X $ and $ Y $ by
$ X\cong Y.$\ The word `operator' will always mean
a bounded linear operator. For any Banach space X, the dual space of bounded
linear functionals on X will be denoted by $ X^{\ast} .$\
Also we use $ \langle x,x^{\ast}\rangle $ or $ x^{\ast}(x) $
for the duality between $ x\in X $ and $ x^{\ast}\in X^{\ast}. $\ We denote the closed unit ball of $ X$ by $ B_{X} $ and the identity
operator on $ X$ is denoted by $ id_{X}. $\
For a bounded linear operator $ T : X \rightarrow Y, $ the adjoint of the operator $ T $
is denoted by $ T^{\ast}. $\ The space of all bounded linear operators, weakly
compact operators, and compact operators from $ X$ to $ Y $ will be denoted by
$ L(X, Y ), W(X, Y ), $ and $ K(X, Y ) ,$ respectively.\ The projective tensor product of two Banach spaces $ X $ and $ Y $ will
be denoted by $ X\widehat{\bigotimes}_{\pi} Y .$\\
A bounded linear operator $ T : X \rightarrow Y $ is called completely continuous,
if $  T$ maps weakly convergent sequences to norm convergent sequences \cite{AlbKal}.\ The set
of all completely continuous operators from $ X $ to $  Y$ is denoted by $ CC(X, Y ). $\
 A bounded linear $ T$ from a Banach space $ X$ to a Banach space $ Y$ is called
Dunford-Pettis completely continuous, if  it transforms Dunford-Pettis and weakly
null sequences to norm null ones.\ The class of Dunford-Pettis completely
continuous operators from $ X $ into $  Y$ is denoted by $ DPcc(X,Y). $\\
 A sequence $ (x_{n})_{n} $ in $ X $ is called weakly $ p $-summable, if $ (x^{\ast}(x_{n}))_{n} \in \ell_{p}$ for each $ x^{\ast}\in X^{\ast}. $\ We denote the set of all weakly $ p $-summable sequences in $ X $ is denoted by $ \ell_{p} ^{w}(X)$ \cite{djt}.\ The weakly $ \infty $-summable sequences are precisely the weakly null sequences.\
Note that, a sequence $(x_{n})_{n}$ in $ X $ is said to be weakly $ p $-convergent to $ x\in X$ if $ (x_{n} - x)_{n}$
is weakly $ p $-summable.\ A sequence $(x_{n})_{n}$ in a Banach space $ X $
is weakly $ p$-Cauchy if for each pair of strictly increasing sequences $(k_{n})_{n}$ and $(j_{n})_{n}$ of positive integers, the sequence $(x_{k_{n}}- x_{j_{n}})_{n}$ is weakly $ p $-summable in $ X $ \cite{ccl}.\ Notice that, every weakly $ p $-convergent sequence is weakly $ p $-Cauchy, and the weakly $ \infty $-Cauchy sequences are precisely the weakly Cauchy sequences.\
A bounded linear operator $ T $ between two Banach spaces is called $ p $-convergent, if it transforms weakly $ p $-summable sequences into norm null sequences \cite{C1}.\ We denote the class of $ p $-convergent operators from $ X $ into $ Y$ by $ C_{p}(X,Y) .$\ A Banach space $ X $ has the $ p $-Schur property (in short $ X\in(C_{p}) $), if  the
identity operator on $ X $ is $ p $-convergent.\
A Banach space $ X$ has the Dunford-Pettis property of order $p$ ($X $ has the $ (DPP_{p})$), if every weakly compact operator on
X is p-convergent. Equivalently, $ X $ has the $(DPP_{p}) $ if and only if for every weakly $ p$-summable sequence $ (x_{n})_n $ in $ X $ and weakly-null sequence $ (x^{\ast}_{n})_n $ in $ X^{\ast},$ we have $x^{\ast}_{n}(x_{n})\rightarrow 0$ as $ n\rightarrow\infty$ \cite{cs}.\\
A subset $ K $ of a Banach space $ X $ is called relatively weakly $p$-compact, if each sequence in $ K $ admits a weakly $ p$-convergent subsequence with limit in $ X .$\ If the ``limit point” of each weakly $ p $-convergent subsequence lies in $ K, $ then we say that $ K $ is a weakly p-compact set.\ Note that, the weakly $ \infty $-compact sets
are precisely the weakly compact.\
A bounded linear operator $ T : X \rightarrow Y $ is called weakly $ p $-compact, if $ T (B_{X}) $ is a relatively weakly $ p $-compact set in $ Y.$\
The set of all weakly
$ p $-compact operators $ T : X \rightarrow Y $ is denoted by $  W_{p}(X, Y ). $\ We refer the reader for undefined terminologies to the
classical references \cite{AlbKal, di1}.\
\section{ p-sequentially Right and p-sequentially Right$ ^{\ast} $ properties\\ on Banach spaces}
 The authors in \cite{ce1, g8} by using Right topology, proved that
a sequence $ (x_{n})_{n} $ in a Banach space $ X$ is Right null if and only if it is Dunford-Pettis and weakly null.\
Also, they showed that a sequence $ (x_{n})_{n} $ in a Banach space $ X$ is Right Cauchy if and only if it is Dunford-Pettis and weakly Cauchy.\\
 Inspired by the above works, for convenience,
we apply the notions  $  p$-Right null and $ p $-Right Cauchy sequences instead of weakly $ p $-summable and weakly $ p $-Cauchy sequences which are Dunford-Pettis sets, respectively.\\
The main aim of this section is to obtain some characterizations of $ p $-Right sets that are relatively weakly $ q $-compact.\
\begin{defn}\label{d1} $ \rm{(i)} $ A bounded subset $ K $ of a Banach space $ X $ is said to be $ p $-Right$ ^{\ast}$ set, if for
every $ p $-Right null sequence $ (x^{\ast}_{n})_{n} $ in $ X^{\ast} $ it follows:
$$ \lim_{n} \sup_{x\in K}\vert x_{n}^{\ast}(x) \vert=0.$$\
$ \rm{(ii)} $ We say that $ X $ has the $ p $-sequentially Right$ ^{\ast} $ property (in short $ X$ has the $ p $-$(SR^{\ast})$ property), if every $ p $-Right$ ^{\ast} $ set is relatively weakly
compact.\
\end{defn}
It is easy to verify that, $ \infty $-Right$ ^{\ast} $ sets are precisely Right$ ^{\ast} $ sets and the $ \infty $-$ (SR)^{\ast} $ property is precisely the sequentially Right$ ^{\ast} $ property (see,  \cite{g6}).

Suppose that $ K $ is a bounded subset of $ X$ and
$ B(K) $ is the Banach space of all bounded real-valued functions
defined on $ K$, provided with the superemum norm.\ The natural evaluation map $ E : X^{\ast}\rightarrow B(K) $ defined by
$ E(x^{\ast})(x) = x^{\ast}(x)$ for every $ x \in K,~~~ x^{\ast} \in X^{\ast},$ has been used by many authors to study properties of $ K.$\ Similarly, if $ K $ is a bounded
subset of $ X^{\ast}, $ the natural evaluation map $ E_{X} :X \rightarrow B(K) $ defined by $ E_{X}(x)(x^{\ast}) =x^{\ast}(x)$ for every $ x \in X,~~~ x^{\ast} \in K,$
(see, \cite{bpo}).\\
At the first, inspired
by Theorem 3.1 of \cite{bpo}, we obtain some characterizations of notions $ p $-Right sets and $ p $-Right$ ^{\ast} $ sets by using evaluation maps
which will be used in the sequel.\

\begin{lem}\label{l1} The following statements hold:\\
$ \rm{(i)} $ If $ T\in L(X,Y),$ then
$ T $ is Dunford-Pettis $ p $-convergent if and only if $ T^{\ast}(B_{Y^{\ast}}) $ is a $ p $-Right subset of $ X^{\ast}. $\\
$ \rm{(ii)} $
A bounded subset $ K $ of $ X^{\ast} $ is a $ p $-Right set if and only if $ E_{X}:X\rightarrow B(K) $ is Dunford-Pettis $ p $-convergent.\\
$ \rm{(iii)} $ If $ T \in L(X,Y) ,$ then
$ T^{\ast} $ is Dunford-Pettis $ p $-convergent if and only if $ T(B_{X}) $ is a $ p $-Right$ ^{\ast} $ subset of $ Y.$ \\
$ \rm{(iv)} $ $ X^{\ast} $ has the $ p $-$ (DPrcP) $ if and only if every bounded subset of $ X $ is a $ p $-Right$ ^{\ast} $ set.\\
$ \rm{(v)} $
A bounded subset $ K $ of $ X$ is a $ p $-Right$ ^{\ast} $ set if and only if $ E:X^{\ast}\rightarrow B(K) $
is Dunford-Pettis $ p $-convergent.\\
$ \rm{(vi)} $ A bounded subset $ K $ of $ X $ is a $ p $-Right$ ^{\ast} $ set if and only if there is a Banach space
$ Y $ and a bounded linear operator $ T:Y\rightarrow X $ so that $ T$ and $ T^{\ast} $ are Dunford-Pettis $ p $-convergent
and $ K\subseteq T(B_{Y}) .$\
\end{lem}
\begin{proof}
$ \rm{(i)} $ Suppose that $ T:X\rightarrow Y $ is a bounded linear operator.\ Clearly, $ T^{\ast}(B_{Y^{\ast}}) $ is a $ p $-Right set if and only if
$$ \lim_{n} \Vert T(x_{n}) \Vert= \lim_{n} (\sup\lbrace \vert \langle y^{\ast}, T( x_{n})\rangle \vert :y^{\ast}\in B_{Y^{\ast}}\rbrace ) $$
$$= \lim_{n} (\sup\lbrace \vert \langle T^{\ast}(y^{\ast}), x_{n}\rangle \vert :y^{\ast}\in B_{Y^{\ast}}\rbrace )=0$$ for each $ p $-Right null sequence $ (x_{n})_{n} $ in $ X$ if and only if $ T $ is Dunford-Pettis $ p $-convergent.\\
$ \rm{(ii)} $ Let $ K $ be a bounded subset of $ X^{\ast}. $\ The evaluation
map $ E_{X}:X\rightarrow B(K) $ is Dunford-Pettis $ p $-converging if and only if $ \Vert E_{X}(x_{n}) \Vert\rightarrow 0$ for each $ p $-Right null sequence $ (x_{n})_{n} $ in $ X $ if and only if
$$ \lim_{n} (\sup \lbrace \vert x^{\ast}(x_{n})\vert :x^{\ast} \in K\rbrace )=0$$ for each $ p $-Right null sequence $ (x_{n})_{n} $ in $ X $ if and only if $K$ is a $ p $-Right set.\\
$ \rm{(iii)} $ Suppose that $ T:X\rightarrow Y $ is a bounded linear operator.\ Clearly, $ T(B_{X}) $ is a $ p $-Right$ ^{\ast} $ set if and only if $$ \lim_{n} \Vert T^{\ast}(y^{\ast}_{n}) \Vert= \lim_{n} (\sup\lbrace \vert \langle x, T^{\ast}( y^{\ast}_{n})\rangle \vert :x\in B_{X}\rbrace ) $$ $$= \lim_{n} (\sup\lbrace \vert \langle T(x), y^{\ast}_{n}\rangle \vert :x\in B_{X}\rbrace )=0$$ for each $ p $-Right null sequence $ (y^{\ast}_{n})_{n} $ in $ X^{\ast}$ if and only if $ T^{\ast} $ is Dunford-Pettis $ p $-convergent.\\
$ \rm{(iv)} $ is clear.\\
$ \rm{(v)} $ Suppose that $ K $ is a bounded subset of $ X $ and $ E:X^{\ast}\rightarrow B(K) $ is a Dunford-Pettis $ p $-convergent operator.\ Thus $ E^{\ast} $ maps the unit ball of $ B(K) ^{\ast},$ to a
$ p $-Right set in $ X^{\ast\ast} .$\ However, if $ k \in K $ and $ \delta_{k}$ denotes the point mass at $ k, $ then
$ E^{\ast}(\lbrace\delta_{k}:k \in K\rbrace) = K, $
and so $ K $ is a $ p $-Right set in $ X^{\ast\ast} .$\ Hence $ K $ is a $ p $-Right$ ^{\ast} $ set in $ X. $\
Conversely, suppose that $ K$ is a $ p $-Right$ ^{\ast} $ set in $ X$, and let $ E:X^{\ast}\rightarrow B(K) $
be the evaluation map.\ If $ (x^{\ast}_{n})_{n} $ is a $ p $-Right null sequence in $ X^{\ast}, $
then
$$\lim_{n}\Vert E(x^{\ast}_{n})\Vert=\lim_{n} (\sup\lbrace \vert x^{\ast}_{n}(x) \vert :x \in K\rbrace)=0, $$
and $ E$ is a Dunford-Pettis $ p$-convergent operator.\\
$ \rm{(vi)} $ Suppose that $ K $ is a $ p $-Right set and $ Y=\ell_{1} (K).$\
Define $ T:Y\rightarrow X $ by $ T(f) =\sum_{k\in K} f(k)k $ for $f \in \ell_{1}(K) .$\ It is clear that
$ T $ is a bounded linear operator, and $ K\subseteq T(B_{\ell_{1}(K)}).$
Since $ \ell_{1} (K)$ has the Schur property, the operator $ T$ is completely continuous and so, it is
$ p $-convergent.\ Thus, $ T $ is a Dunford-Pettis $ p $-convergent operator.\ It is easy to verify that, $ T^{\ast} $ is the
evaluation map $ E:X^{\ast} \rightarrow B(K).$\ Hence, $ T^{\ast} $ is Dunford-Pettis $ p $-convergent by $\rm{ (v)}. $\
\end{proof}
Recall that, a bounded subset $ K $ of $ X^{\ast} $ is called an $ (L) $ set, if each weakly null sequence
$ (x_{n})_{n} $ in $  X$ tends to $ 0 $ uniformly on $ K $ \cite{AlbKal}.\\
Bator et al.\ \cite{bpo} showed that every $ (L) $ subset of $ X^{\ast} $ is a Dunford-Pettis set in $ X^{\ast} $ if and only if $ T^{\ast\ast} $ is completely continuous whenever $ Y $ is an arbitrary Banach space and $ T:X\rightarrow Y $ is a completely continuous operator.\\
It is easy to verify that, for each $ 1\leq p\leq \infty, $ every $ p $-Right$ ^{\ast} $ subset of dual space is a $ p $-Right set, while the converse of implication is false.\
The following theorem continues our study of the relationship between
$ p $-Right subsets and $ p $-Right$ ^{\ast} $ subsets of dual spaces.
\begin{thm}\label{t1} Every $ p$-Right subset of $ X^{\ast} $ is a $ p $-Right$ ^{\ast} $ set in $ X^{\ast} $ if and only if
$ T^{\ast\ast} $ is a Dunford-Pettis $ p$-convergent operator whenever $ Y $ is an arbitrary Banach space and
$ T : X \rightarrow Y $ is a Dunford-Pettis $ p $-convergent operator.
\end{thm}
\begin{proof}
Suppose that
$ T:X\rightarrow Y $ is a Dunford-Pettis $ p $-convergent operator.\ The part (i) of Lemma of \ref{l1}, yields that $ T^{\ast}(B_{Y^{\ast}} ) $ is a
$ p $-Right set.\ By the hypothesis $ T^{\ast}(B_{Y^{\ast}} ) $ is a $ p $-Right$ ^{\ast} $ set.\ By applying the Lemma \ref{l1} $ \rm{(iii)}, $ we see that $ T^{\ast\ast} $ is a Dunford-Pettis $p$-convergent operator.\
Conversely, suppose that
$ K $ is a $ p $-Right subset of $ X^{\ast} .$\ The part (ii) of Lemma \ref{l1}, implies that $ E_{X} $ is Dunford-Pettis $ p $-convergent.\ Therefore, by the hypothesis, $ E_{X}^{\ast\ast} $ is Dunford-Pettis $ p$-convergent.\ Hence, if the unit ball of $ B(K)^{\ast} $ denoted by $ S, $ then
$ E^{\ast}_{X}(S) $ is a $ p $-Right$ ^{\ast} $ set.\ Since $K \subset E_{X}^{\ast}(S), $
$ K $ is a $ p $-Right$ ^{\ast} $ set in $ X^{\ast}. $\
\end{proof}
\begin{cor}\label{c1}
Every Right subset of $ X^{\ast} $ is a Right$ ^{\ast} $ set in $ X^{\ast} $ if and only if
$ T^{\ast\ast} $ is Dunford-Pettis completely continuous whenever $ Y $ is an arbitrary Banach space and
$ T : X \rightarrow Y $ is a pseudo weakly
compact operator.
\end{cor}
Recall from \cite{AlbKal}, that the space of all finite regular Borel signed measures on the
compact space K is denoted by M(K). It is well known that $ M(K) = C(K)^{\ast}. $
\begin{cor}\label{c2}
If $ K $ is a compact Hausdorff space, then every $ p $-Right subset
of $ M(K)$ is a $ p $-Right$ ^{\ast} $ set in $ M(K) .$\
\end{cor}
\begin{proof} Suppose that $ K $ is a compact Hausdorff space, $ Y $ is a Banach space and $ T : C(K) \rightarrow Y $ is a Dunford-Pettis $ p $-convergent operator.\ Since $ C(K) $ has the $ p $-sequentially Right property, $ T $ is weakly compact and so, $ T^{\ast\ast} $ is weakly compact.\ Therefore, $ T^{\ast\ast} $ is Dunford-Pettis $ p $-convergent.\ Hence, Theorem \ref{t1} implies that, every $ p $-Right subset of $ M(K)$ is a $ p $-Right$ ^{\ast} $ set in $ M(K).$
\end{proof}

If $ \mathcal{M} $ is a closed subspace of $ \mathcal{U}(X, Y ), $ then for arbitrary
elements $ x \in X $ and $ y^{\ast} \in Y^{\ast} ,$ the evaluation operators $ \phi_{x} :\mathcal{M}\rightarrow Y$ and
$ \psi_{y^{\ast}} $ on $ \mathcal{M} $ are defined by $\phi_{x}(T) = T(x), $ $ \psi_{y^{\ast}}(T)=T^{\ast}(y^{\ast}) .$ Also,
the point evaluation sets related to $ x \in X $ and $ y^{\ast} \in Y^{\ast} $ are the images of the
closed unit ball $ B_{\mathcal{M}} $ of $ \mathcal{M} ,$ under the evaluation operators $ \phi_{x} $ and $ \psi_{y^{\ast}}$ and they
are denoted by $ \mathcal{M}_{1}(x) $ and $\widetilde{\mathcal{M}} _{1}(y^{\ast}) $ respectively \cite{MZ}.\ Note that, if we speak about $ \mathcal{U} (X,Y)$ or its linear subspace $\mathcal{M}, $ then the related norm is the ideal norm $ \mathcal{A}(.) $ while, the operator norm $ \Vert .\Vert $ is applied when the space is a linear subspace of $L(X,Y) . $\
\begin{thm}\label{t2} Suppose that $ 1\leq p\leq \infty $ and the dual $ \mathcal{M^{\ast}} $ of a
closed subspace $ \mathcal{M}\subseteq \mathcal{U}(X, Y ) $ has the $ p $-$ (DPrcP) .$\ Then all of the point
evaluations $ \mathcal{M}_{1}(x) $ and $\widetilde{\mathcal{M}} _{1}(y^{\ast})$ are $ p $-Right sets.
\end{thm}
\begin{proof}
Since $ \mathcal{M^{\ast}} $ has the $ p $-$ (DPrcP),$ $ \phi_{x} ^{\ast}$
is a Dunford-Pettis $ p $-convergent operator.\ Now, suppose that $ (y_{n}^{\ast})_{n} $ is a
$ p $-Right null sequence in $ Y^{\ast}. $\ It is clear that $ \displaystyle\lim_{n\rightarrow\infty} \Vert \phi_{x}^{\ast}(y_{n}^{\ast})) \Vert =0,$ for all $ x\in X. $\ On the other hand,
$$ \Vert \phi_{x}^{\ast}(y_{n}^{\ast})) \Vert=\sup\lbrace \vert \phi_{x}^{\ast}y_{n}^{\ast} (T ) \vert : T\in B_{\mathcal{M}} \rbrace=\sup\lbrace \vert y_{n}^{\ast} (T(x)) \vert : T\in B_{\mathcal{M}} \rbrace .$$\
This shows that $ \mathcal{M}_{1}(x) $ is a $ p $-Right set in $ Y, $ for all $ x\in X. $\ A similar proof shows
that $\widetilde{\mathcal{M}} _{1}(y^{\ast})$ is a a $ p $-Right set in $ X^{\ast}. $
\end{proof}

In the following, we obtain some sufficient conditions for which the point evaluations $ \mathcal{M}_{1}(x) $ and $\widetilde{\mathcal{M}} _{1}(y^{\ast})$ are relatively weakly compact for all $ x \in X $ and all $ y^{\ast} \in Y^{\ast}. $\
\begin{thm}\label{t3}
Suppose that $ 1\leq p\leq \infty $ and $ X^{\ast\ast} $ and $ Y^{\ast} $ have the $ p $-$ (SR)$ property.\ If $\mathcal{M}\subseteq \mathcal{U} (X,Y) $ is a closed subspace so that the natural restriction operator $ R: \mathcal{U} (X,Y)^{\ast}\rightarrow \mathcal{M^{\ast}} $
is a Dunford-Pettis $ p $-convergent operator, then all of the point evaluations $ \mathcal{M}_{1}(x) $ and $\widetilde{\mathcal{M}} _{1}(y^{\ast})$ are relatively weakly compact.\
\end{thm}
\begin{proof}
It is enough to show
that $ \phi_{x} $ and $ \psi_{y^{\ast}} $ are weakly
compact operators.\ For this purpose suppose that $ T\in \mathcal{U} (X,Y).$\ Since $ \Vert T \Vert\leq \mathcal{A}(T), $ it is not difficult to show that,
the operator $ \psi:X^{\ast\ast}\widehat{\bigotimes}_{\pi}Y^{\ast}\rightarrow \mathcal{U} (X,Y)^{\ast} $ which is defined by
$$ \vartheta\mapsto tr(T^{\ast\ast}\vartheta) =\displaystyle \sum_{n=1}^{\infty}\langle T^{\ast\ast} x^{\ast\ast} _{n}, y_{n}^{\ast} \rangle$$ is linear and continuous, where $\vartheta=\displaystyle \sum_{n=1}^{\infty} x_{n} ^{\ast\ast}\bigotimes y_{n}^{\ast}.$\ Fix now an arbitrary element $ x\in X $ we define $ U_{x}:Y^{\ast}\rightarrow X^{\ast\ast}\widehat{\bigotimes}_{\pi}Y^{\ast}$ by $ U_{x}(y^{\ast})=x\bigotimes y^{\ast}. $\
It is clear that the operator $ \phi_{x}^{\ast}= R\circ \psi \circ U_{x}$ is a Dunford-Pettis
$ p$-convergent.\
Since $ Y^{\ast}$ has the $ p $-$ (SR) $ property, we conclude that
$ \phi_{x}^{\ast}$ is a weakly
compact operator.\ Hence, $ \phi_{x} $ is weakly
compact.\ Similarly, we can see that $ \psi_{y^{\ast}} $ is weakly
compact.\
\end{proof}

\begin{thm}\label{t4}
Let $ X $ be a Banach space and $ 1\leq p < q\leq \infty. $\ The following statements are equivalent:\\
$\rm{(i)}$ For every Banach space $ Y, $ if $ T : X\rightarrow Y $ is a Dunford-Pettis $ p$-convergent operator, then $ T$ has a weakly $ q $-precompact (weakly $ q $-compact, $ q $-compact) adjoint,\\
$\rm{(ii)}$ Same as $\rm{(i)}$ with $ Y=\ell_{\infty} ,$\\
$\rm{(iii)}$ Every $ p$-Right subset of $ X^{\ast} $ is weakly $ q $-precompact (relatively weakly $q $-compact, $ q $-compact).
\end{thm}
\begin{proof}
We will show that in the relatively weakly
$ q $-compact case.\ The other
proof is similar.\\
(i) $ \Rightarrow $ (ii) is obvious.\\
(ii) $ \Rightarrow$ (iii) Let $ K$ be a $ p$-Right subset of $ X^{\ast} $ and let $ (x^{\ast}_{n})_{n} $
be a sequence
in $ K. $\ Define $ T:X\rightarrow \ell_{\infty} $ by $ T(x)=(x^{\ast}_{n}(x)) .$\
Let $ (x_{n})_{n} $ be a $ p $-Right null sequence in $ X. $\ Since $ K $ is a $ p$-Right set,
$$\lim_{n\rightarrow \infty}\Vert T(x_{n})\Vert=\lim_{n\rightarrow \infty}\sup_{m} \vert x^{\ast}_{m}(x_{n})\vert=0. $$
Therefore, $ T $ is Dunford-Pettis $ p$-convergent.\ Hence, by the hypothesis, $ T^{\ast} $ is weakly $ q $-compact and so,
$ (T^{\ast}(e^{1}_{n}))_{n}=(x^{\ast}_{n})_{n}$
has a weakly $ q$-convergent subsequence.\\
(iii) $ \Rightarrow $ (i) Let $ T:X\rightarrow Y $ be a Dunford-Pettis $ p $-convergent operator.\ Then $ T^{\ast}(B_{Y^{\ast}}) $
is a $ p$-Right subset of $ X. $\ Therefore $ T^{\ast}(B_{Y^{\ast}}) $ is relatively weakly $ q $-compact,
and thus $ T^{\ast}$ is weakly $ q $-compact.\
\end{proof}
Let $ A $ and $ B $ be nonempty subsets of a Banach space $ X, $ we define ordinary distance and non-symmetrized Hausdorff distance respectively, by
\begin{center}
$ d(A,B)=\inf\lbrace d(a,b):a\in A, b\in B \rbrace, \hspace{.9 cm} \hat{d}(A,B) =\sup\lbrace d(a,B): a \in A\rbrace.$\end{center}
Let $ X $ be a Banach space and $ K $ be a bounded subset of $ X^{\ast}. $\ For $ 1\leq p \leq \infty, $
we set
\begin{center}
$ \zeta_{p}(K)=\inf\lbrace \hat{d}(A,K) : K\subset X^{\ast}$ is a $ p $-Right set $ \rbrace. $
\end{center}
We can conclude that $ \zeta_{p}(K)=0 $ if and only if $K\subset X^{\ast}$ is a $ p $-Right set.\ Now,
let $ K$ be a bounded subset of a Banach space $ X. $\ The de Blasi measure
of weak non-compactness of $ K$ is defined by
\begin{center}
$\omega(K) =\inf\lbrace \hat{d}(K,A):\emptyset \neq A\subset X $ is weakly-compact $ \rbrace. $
\end{center}
It is clear that
$ \omega(K) =0$ if and only if $ K $ is relatively weakly compact.\
For a bounded linear operator
$ T : X \rightarrow Y , $ we denote $ \zeta_{p} (T(B_{X})), \omega(T(B_{X}) $ by $ \zeta_{p} (T), \omega(T)$ respectively.\
\begin{cor}\label{c3}
Let $ X $ be a Banach space and $ 1\leq p \leq \infty. $\ The following statements are equivalent:\\
$ \rm{(i)} $ $ X $ has the $ p$-$ (SR)$ property.\\
$ \rm{(ii)} $ For each Banach space $ Y, $ adjoint every Dunford-Pettis $ p $-convergent $ T:X\rightarrow Y $
is weakly compact.\\
$ \rm{(iii)} $ Same as $ \rm{(ii)} $ with $ Y=\ell_{\infty}. $\\
$ \rm{(iv)} $ $ \omega (T^{\ast}) \leq \zeta_{p}(T^{\ast})$ for every bounded linear operator $ T$ from $ X$ into any Banach space $ Y. $\\
$ \rm{(v)} $ $ \omega (K)\leq \zeta_{p} (K)$ for every bounded subset $ K $ of $ X^{\ast} .$
\end{cor}
Recently, the notions of $ p $-$ (V) $ sets and $ p $-$ (V) $ property as an extension of the notions $ (V) $ sets and Pelczy\'{n}ski's property $ (V) $ introduced by Li et al.\ \cite{ccl1} as follows:\
\begin{itemize}
\item A bounded subset $ K $ of $ X^{\ast} $ is a $ p $-$ (V ) $ set, if $ \displaystyle \lim_{n\rightarrow\infty}\displaystyle \sup_{x^{\ast}\in K}\vert x^{\ast}(x_{n}) \vert =0,$
for every weakly $ p $-summable sequence $ (x_{n})_{n} $ in $ X. $\
\item A Banach space $ X $ has
Pelczy\'{n}ski's property $ (V) $ of order $ p$ ($p$-$ (V) $ property), if every $ p $-$ (V) $ subset of $
X^{\ast}$ is relatively weakly compact.
\item A bounded subset $ K $ of $ X$ is a $ p $-$ (V^{\ast} ) $ set, if $ \displaystyle \lim_{n\rightarrow\infty}\displaystyle \sup_{x\in K}\vert x^{\ast}_{n}(x) \vert =0,$
for every weakly $ p $-summable sequence $ (x^{\ast}_{n})_{n} $ in $ X^{\ast}. $\
\item A Banach space $ X $ has
Pelczy\'{n}ski's property $ (V^{\ast}) $ of order $ p$ ($p$-$ (V^{\ast}) $ property), if every $ p $-$ (V^{\ast}) $ subset of $
X$ is relatively weakly compact.
\end{itemize}

The proof of the following proposition is similar to the proof {\rm (\cite[Corollary 3.19 (ii) ]{g9})}.\ Therefore, its proof is omitted
\begin{prop} \label{p1}
$ X^{\ast} $ has the $ (DPP_{p}) $ if and only if each $ p $-Right$ ^{\ast} $ set in $ X $ is a $ p $-$ (V^{\ast}) $ set.
\end{prop}
\begin{prop}\label{p2}\rm
Let $ X $ be a Banach space.\ The following statements hold:\\
$ \rm{(i)} $ If $ X $ has the $ p $-$ (SR) $ property, then $ X^{\ast} $ has the $ p $-$ (SR^{\ast})$ property.\\
$ \rm{(ii)} $ If $ X $ has the $ p $-$ (SR) $ property, then $ X $ has the $ p $-$ (V)$ property.\\
$ \rm{(iii)} $ If $ X^{\ast} $ has the $ p $-$ (SR) $ property, then $ X $ has the $ p $-$ (SR^{\ast}) $ property,\\
$ \rm{(iv)} $ If $ X $ has the $ p $-$ (SR^{\ast}) $ property, then $ X $ has the $ p $-$ (V^{\ast}) $ property.\\
$ \rm{(v)} $ Let $ Y$ be a reflexive subspace of $ X. $\ If $ \frac{X}{Y} $ has the $ p $-$ (SR^{\ast}) $ property,
then $ X $ has the same property.\
\end{prop}
\begin{proof}
Since, the proof of $\rm{(i)}, $ $\rm{(ii)} ,$ $ \rm{(iii)} $ and $ \rm{(iv)} $ are clear, we only prove $ \rm{(v)} .$\\
$ \rm{(i)} $ Let $ Q : X \rightarrow \frac{X}{Y} $ be the quotient map.\ Let $ K$ be a $ p $-Right$ ^{\ast}$ set in $ X $ and $ (x_{n})_{n} $ be
an arbitrary sequence in $ K. $\ Then $ (Q(x_{n}))_{n} $ is a $ p $-Right$ ^{\ast} $ set in $ \frac{X}{Y},$ and thus relatively weakly
compact.\ By passing to a subsequence, suppose $ (Q(x_{n}))_{n} $ is weakly convergent.\ By
{\rm (\cite[Theorem 2.7]{gh})}, $ (x_{n})_{n} $ has a weakly convergent subsequence.\ Thus $ X$ has the $ p $-$ (SR^{\ast}) $ property.\
\end{proof}
\begin{cor}\label{c4} If
$ X $ has the $ (DPP_{p}),$ then $ X $ has Pelczy$\acute{n} $ski's property $ (V ) $ of order $p $ if and only $ X $ has the $ p $-$ (SR) $ property.
\end{cor}
\begin{proof}Suppose that $ X $ has Pelczy$ \acute{n} $ski's property $ (V ) $ of order $ p. $\ We show that for each Banach space $ Y, $ adjoint every Dunford-Pettis $ p $-convergent $ T:X\rightarrow Y $
is weakly compact.\ Let $ T \in DPC_{p}(X,Y) .$\ The part $ \rm{(i)} $ of Lemma \ref{l1}, implies that $ T^{\ast}(B_{Y^{\ast}}) $ is a $ p $-Right set in $ X^{\ast}. $\ So, {\rm (\cite[Corollary 3.19 (ii) ]{g9})} implies that $ T^{\ast}(B_{Y^{\ast}}) $ is a $ p $-$ (V) $ set in $ X^{\ast}. $\ Since $ X $ has Pelczy$ \acute{n} $ski's property $ (V ) $ of order $ p, $ $ T^{\ast} $ is weakly compact.\ Hence, Corollary \ref{c3} implies that $X $ has the $ p $-$ (SR)$ property.
Conversely, If $ X $ has the $ p $-$ (SR)$ property, then $ X $ has Pelczy$\acute{n} $ski's property $ (V ) $ of order $ p.$\ Since every $ p $-$ (V) $ set in $ X^{\ast} $ is a $ p $-Right set.
\end{proof}
Suppose that $ X $ is a Banach space and $ Y $ is a subspace of $ X^{\ast}. $\ We define $^{\bot}Y:=\lbrace x\in X: y^{\ast}(x)=0 ~~$ for all $~~~ y^{\ast}\in Y^{\ast} \rbrace .$\
\begin{cor}\label{c5}
$ \rm{(i)} $ If $ X$ is an infinite dimensional non reflexive Banach space with the $ p$-Schur property, then $ X$
does not have the $ p $-$ (SR)$ property.\\
$ \rm{(ii)} $ If every separable subspace of $ X$ has the $ p $-$ (SR) $ property, then $ X $ has the same property.\\
$ \rm{(iii)} $ Let $Y $ be a reflexive subspace of $ X^{\ast}. $\ If $ ^{\bot}Y$ has the $ p $-$ (SR) $ property, then $ X $ has the same property.\
\end{cor}
\begin{proof}
$ \rm{(i)} $ Since $ X\in C_{p},$ the identity operator $ id_{X} : X \rightarrow X $ is $ p $-convergent and so, it is
Dunford-Pettis $ p $-convergent.\ It is clear that $ id_{X} $ is not weakly compact.\ Hence, Corollary \ref{c3} implies that $ X $ does not have the $ p $-$ (SR)$ property.\\
$ \rm{(ii)} $ Let $ (x_{n})_{n} $ be a sequence in $ B_{X} $ and let $ Z = [x_{n} : n \in \mathbb{N}] $ be the closed linear
span of $ (x_{n})_{n} .$\ Since $ Z$ is a separable subspace of $ X, $ $ Z $ has the $ p $-$(SR)$ property.\ Now, let $ T : X \rightarrow Y $ be a Dunford-Pettis $ p $-convergent operator.\ It is clear that $ T_{\vert Z} $ is a Dunford-Pettis $ p $-convergent operator.\ Therefore, Corollary \ref{c3}, implies that $ T_{\vert Z} $ is weakly compact.\ Hence,
there is a subsequence $ (x_{n_{k}} )_{k} $ of $ (x_{n})_{n} $ so that $ T(x_{n_{k}} ) $ is
weakly convergent.\ Thus $ T$ is weakly compact.\ Now an appeal to Corollary \ref{c3} completes the proof.\\
$ \rm{(iii)} $ By {\rm (\cite[Theorem 1.\ 10.\ 6]{m})},
there exists a quotient map $ Q: X^{\ast}\rightarrow \frac{X^{\ast}}{Y} $ and a surjective isomorphism $ i:\frac{X^{\ast}}{Y}\rightarrow ~(^{\bot}Y) ^{\ast}$ such that $ i\circ Q:X^{\ast}\rightarrow ~(^{\bot}Y)^{\ast} $ is $ w^{\ast} $-$ w^{\ast} $ continuous.\ Therefore, there is $ S: ~^{\bot}Y \rightarrow X $
with $ S^{\ast}=i\circ Q. $\ Hence, for any Dunford-Pettis $ p $-convergent operator $ ~~~ $ $ T:X\rightarrow Z, $
the operator $ T\circ S: ~^{\bot}Y \rightarrow Z$ is Dunford-Pettis $ p$-convergent, that must
be weakly compact; hence, $ S^{\ast}\circ T^{\ast} =i\circ Q\circ T^{\ast} $ is also weakly compact,
this in turn gives that $ Q\circ T^{\ast} $ must be weakly compact, since $ i$ is a surjective
isomorphism.\ Therefore
$ T^{\ast} .$ \ The Corollary \ref{c3} completes the proof.\
\end{proof}

\begin{thm}\label{t5}
$ \rm{(i)}$ Suppose that $L_{w^{\ast}}(X^{\ast} ,Y)= K_{w^{\ast}}(X^{\ast} ,Y) .$\ If $ X$ and $ Y $ have the $ p $-$ (SR)$ property, then $ K_{w^{\ast}}(X^{\ast} ,Y) $ has the same property.\\
$ \rm{(i)}$ Suppose that $L(X ,Y)= K(X ,Y) .$\ If $ X^{\ast} $ and $ Y $ have the $ p $-$ (SR)$ property, then $ K(X, Y ) $ has the same property.
\end{thm}
\begin{proof} Since the proofs of (i) and (ii) are essentially the same, we only present
that of (i).\\
$ \rm{(i)}$ Suppose $ X $ and $ Y $ have the $ p $-$ (SR) $ property.\ Let $ H$ be a
$ p $-Right subset of $ K_{w^{\ast}}(X^{\ast} ,Y). $\
For fixed $ x^{\ast}$ in $ X^{\ast}, $ the map $ T\mapsto T(x^{\ast}) $ is a bounded
operator from $ K_{w^{\ast}}(X^{\ast} ,Y) $ into $ Y. $\ It can easily seen that continuous linear images of $ p $-Right sets are $ p $-Right sets.\ Therefore, $ H(x^{\ast}) $
is a $ p $-Right subset of $ Y, $ hence it is relatively
weakly compact.\ For fixed $ y^{\ast}$ in $ Y^{\ast} $ the map $ T\mapsto T^{\ast}(y^{\ast}) $
is a bounded linear operator from
$ K_{w^{\ast}}(X^{\ast} ,Y) $ into $ X. $\ So, $ H^{\ast}(y^{\ast}) $
is a $ p $-Right subset of $ X, $ hence it is relatively weakly compact.\ Then,
{\rm (\cite[Theorem 4.\ 8]{g17})} implies that $ H $ is relatively weakly compact.
\end{proof}
Cilia and Emmanuele in \cite{ce1}
investigated whether
the projective tensor product of two Banach spaces $ X $ and $ Y $ has the sequentially Right
property when $  X$ and $ Y $ have the respective property.\\
In the following,
the stability of $ p$-sequentially Right property for projective tensor product between Banach spaces is investigated.
\begin{thm}\label{t6}
Suppose that $ X$ has the $ p $-$ (SR) $ property and $ Y $ is a reflexive space.\ If $ L(X,Y^{\ast}) = K(X,Y^{\ast}),$
then
$ X \widehat{\bigotimes}_{\pi} Y $ has the $ p $-$(SR)$ property.\
\end{thm}
\begin{proof}
Let $ H $ be a $ p $-Right subset of $ (X \widehat{\bigotimes}_{\pi} Y)^{\ast}\simeq L(X,Y^{\ast}). $\ We claim that $ K $
is relatively weakly compact. We show that the conditions (i) and (ii) of {\rm (\cite[Theorem 4]{g20})} are true.\
Let $ (T_{n}) $ be an arbitrary sequence in $ H .$\ If $ x\in X, $ it is enough to
show that
$ \lbrace T_{n}(x) : n\in \mathbb{N} \rbrace$
is a $ p $-Right subset of $ Y^{\ast}. $\ For this purpose, suppose
that $ (y_{n})_{n} $ is a $ p $-Right null sequence in $ Y. $\ For each $ n\in \mathbb{N}, $ we have:
$$ \langle T_{n} (x) ,y_{n} \rangle=\langle T_{n}, x\otimes y_{n}\rangle .$$
We claim that $ (x \otimes y_{n})_{n}$ is a $ p $-Right null sequence in $ X \widehat{\bigotimes}_{\pi} Y. $\ If $ T\in (X \widehat{\bigotimes}_{\pi} Y)^{\ast}\simeq L(X,Y^{\ast}), $ then
$$ \vert\langle T, x\otimes y_{n}\rangle\vert=\vert \langle T(x), y_{n} \rangle\vert \in \ell_{p},$$
since $ (y_{n}) _{n}$ is weakly $ p $-summable.\ Thus $ (x \otimes y_{n})_{n}$ is weakly $ p $-summable in $ X \widehat {\bigotimes}_{\pi} Y. $\
Let $ (A_{n})_{n} $ be a weakly null sequence in $ (X \widehat {\bigotimes}_{\pi} Y)^{\ast}\simeq L(X,Y^{\ast}). $\ Since the map $\varphi_{x} : L(X,Y^{\ast})\rightarrow Y^{\ast},$
$\varphi_{x}(T) =T(x)$
is linear and bounded, $ (A_{n}(x))_{n} $ is weakly null in $ Y^{\ast}. $\ Since $ (y_{n})_{n} $ is a Dunford-Pettis sequence in $ Y. $\
$$ \vert \langle A_{n}, x\otimes y_{n}\rangle \vert =\vert \langle A_{n}(x), y_{n}\rangle\vert\rightarrow 0.$$
Hence, $ (x\otimes y_{n})_{n} $ is Dunford-Pettis and so, $ (x\otimes y_{n})_{n} $ is
$ p $-Right null in $ X \widehat {\bigotimes}_{\pi} Y.$\ Therefore, the equivalence $ \rm{(i)} $ and $ \rm{(v)} $ in {\rm (\cite[Theorem 3.26]{g9})} implies that
$ \lbrace T_{n}(x) :n\in \mathbb{N}\rbrace$ is a $ p $-Right set in $ Y^{\ast}. $\ Therefore,
$ \lbrace T_{n}(x) :n\in \mathbb{N}\rbrace$ is a relatively weakly compact.\ Now, let $ y\in Y$ and $ (x_{n})_{n} $ be a
$ p $-Right null sequence in $ X. $\ An argument similar to the above one can see that $ (x_{n}\otimes y)_{n} $ is a $ p $-Right null sequence in
$ X \widehat {\bigotimes}_{\pi} Y. $\ Therefore, by reapplying {\rm (\cite[Theorem 3.26]{g9})}
$ \lbrace T^{\ast}_{n} (y):n\in \mathbb{N}\rbrace$ is a $ p $-Right subset of $ X^{\ast}.$\ So, $ \lbrace T^{\ast}_{n} (y):n\in \mathbb{N}\rbrace$ is relatively weakly compact.\ Hence $ H $ is relatively weakly compact.\
\end{proof}
Let $ (X_{n})_{n\in \mathbb{N}} $ be a sequence of Banach spaces.\ If $ 1\leq r <\infty $ the space of all vector-valued sequences $ (\displaystyle\sum_{n=1}^{\infty}\oplus X_{n})_{\ell_{r}} $ is called, the infinite direct sum of $ X_{n} $ in the sense of $ \ell_{r}, $ consisting of all sequences $ x=(x_{n})_{n} $ with values in $X_{n} $ such that $ \Vert x \Vert_{r}=(\displaystyle\sum_{n=1}^{\infty}\Vert x_{n}\Vert^{r} )^{\frac{1}{r}}<\infty.$\ For every $ n\in \mathbb{N},$ we
denote the canonical projection from $ (\displaystyle\sum_{n=1}^{\infty}\oplus X_{n})_{\ell_{r}} $
into $ X_{n} $ by $ \pi_{n} .$\ Also, we denote the canonical
projection from $ (\displaystyle\sum_{n=1}^{\infty}\oplus X^{\ast}_{n})_{\ell_{r^{\ast}}} $
onto $ X^{\ast}_{n} $
by $ P_{n} .$\\

Using the {\rm (\cite[Corollary 3.19]{g9})},and {\rm (\cite[Theorem 3.1]{ccl1})}, we obtain the following result:
\begin{thm}\label{t7}
Let $ 1 < p<\infty $ and $ (X_{n})_{n} $ be a sequence of Banach spaces with $ (DPP_{p}) $ and let $ X=(\sum_{n=1}^{\infty}\oplus X_{n})_{\ell_{p}}.$\
The following are equivalent
for a bounded subset $ K$ of $ X^{\ast}: $\\
$ \rm{(i)} $ $ K $ is a $ p^{\ast} $-Right set.\\
$ \rm{(ii)} $ $ P_{n} (K)$ is a $ p^{\ast} $-Right set for each $ n\in \mathbb{N} $ and
$$ \lim_{n\rightarrow\infty}\sup\lbrace \sum_{k=n}^{\infty}\Vert P_{k}x^{\ast}\Vert^{p^{\ast}} : x^{\ast}\in K\rbrace=0.$$
\end{thm}
\begin{thm}\label{t8}
Let $ 1 < p<\infty $ and $ (X_{n})_{n} $ be a sequence of Banach spaces.\ If $ X=(\displaystyle\sum_{n=1}^{\infty}\oplus X_{n})_{\ell_{p}} $
and $ 1 \leq q < p^{\ast} ,$
then a bounded subset $ K $ of $ X^{\ast} $ is
$ q $-Right set if and only if each $ P_{n}(K) $ is $ q $-Right set.
\end{thm}
\begin{proof} It can easily seen that
continuous linear images of $ q $-Right set is $ q$-Right set.\ Therefore, we only prove the sufficient part.\ Assume that $ K $ is not a $ q $-Right set.\ Therefore, there exist $ \varepsilon_{0}>0, $ a $ q $-Right null sequence $ (x_{n})_{n} $ in $ X $ and a sequence
$ (x^{\ast}_{n})_{n} $ in $ K $ such that
\begin{center}

$ \vert \langle x^{\ast}_{n}, x_{n} \rangle \vert=\vert \displaystyle\sum _{k=1}^{\infty} \langle P_{k} x^{\ast}_{n} , \pi_{k} x_{n} \rangle \vert > \varepsilon_{0},~~~~~~~~~~~~~~n =1,2,3,...~~~~~~~~~~$ $ ~~~~~~~~~~~~~(\ast )$
\end{center}
By the assumption, we obtain
\begin{center}
$\displaystyle\lim_{n\rightarrow \infty} \vert \langle P_{k} x^{\ast}_{n} ,\pi_{k} x_{n} \rangle \vert=0 , ~~~~~~~~ k =1,2,3,...~~~~$ $ ~~~~~~ (\ast\ast) $
\end{center}

By induction on $ n $ in $ (\ast) $ and $ k$ in $ (\ast\ast), $ there exist two
strictly increasing sequences $ (n_{j})_{j} $ and $ (k_{j} )_{j}$ of positive integers such that
\begin{center}
$ \vert \displaystyle\sum_{k=k_{j-1}+1}^{k_{j}} \langle P_{k} x^{\ast}_{n_{j}} , \pi_{k} x_{n_{j}} \rangle\vert >\frac{\varepsilon_{0}}{2}, ~~~~~~j=1,2,3,...$
\end{center}
For each $ j = 1, 2, ..., $ we consider $ y_{j}=x_{n_{j}} $ and $ y^{\ast}_{j} \in X^{\ast}$
by

\begin{equation*}
P_{k}y^{\ast}_{j}=
\begin{cases}
P_{k_{j}}x^{\ast}_{n_{j}} & \text{if } k_{j-1}+1 \leq k\leq k_{j},\\
0 & \text{otherwise }.
\end{cases}
\end{equation*}
It is clear that $ ( y_{j})_{j} $ is a $q $-Right null sequence in $ X $ such that
\begin{center}
$ \vert\langle y^{\ast}_{j} ,y_{j}\rangle\vert=\vert \displaystyle\sum_{k=k_{j-1}+1}^{k_{j}} \langle P_{k} x^{\ast}_{n_{j}} , \pi_{k} x_{n_{j}} \rangle\vert >\frac{\varepsilon_{0}}{2}, ~~~~~~j=1,2,3,... $
\end{center}
Since the sequence $ ( y^{\ast}_{j})_{j} $
has pairwise disjoint supports, Proposition 6.4.1 of \cite{AlbKal} implies that $ ( y^{\ast}_{j})_{j} $
is equivalent to the unit vector basis $ ( e^{p^{\ast}}_{j})_{j} $ of $ \ell_{p^{\ast}} .$\ Suppose that $ R $ is an isomorphic
embedding from $ \ell_{p^{\ast}} $ into $ X^{\ast} $ such that $ R( e^{p^{\ast}}_{j})=y^{\ast}_{j} (j = 1, 2, ...).$\
Now, let $ T $ be a bounded linear
operator from $\ell_{q^{\ast}} $ into $ X. $\ By Pitt’s Theorem \cite{AlbKal}, the operator $ T^{\ast}\circ R $ is compact
and so the sequence $ (T^{\ast}(y^{\ast}_{j}))_{j}=(T^{\ast}R(e^{\ast}_{j}))_{j} $
is relatively norm compact.\ Hence, Theorem 2.3 of \cite{ccl1} implies that the sequence $ (y^{\ast}_{j}) _{j}$ is a $ q $-$ (V) $ set and so is a
$ q $-Right set.\ Since $ (y_{j})_{j} $
is $ q $-Right null, we have
\begin{center}
$ \vert \langle y^{\ast}_{n}, y_{n}\rangle\vert\leq \sup_{j}\vert \langle y^{\ast}_{j}, y_{n}\rangle\vert\rightarrow 0$ as $ ~~~~~~n\rightarrow \infty, $
\end{center}
which is a contradiction.
\end{proof}
\begin{thm}\label{t9}
Let $ (X_{n})_{n} $ be a sequence of Banach spaces.\ If $ 1 < r <\infty $ and $ 1 \leq p <\infty ,$
then $X=(\displaystyle\sum_{n=1}^{\infty}\oplus X_{n})_{\ell_{r}} $ has the $ p $-$ (SR) $ property if and only if each $ X_{n} $ has the same property.
\end{thm}
\begin{proof} It is clear that if $ X $ has the $ p$-$ (SR) $ property, then each $ X_{n} $ has the $ p $-$ (SR)$ property.\ Conversely,
let $ K$ be a $ p $-Right subset of $ X^{\ast}. $\ Since continuous linear images of $ p $-Right sets are $ p $-Right sets, each $ P_{n}(K) $ is
also a $ p $-Right set.\ Since $ X_{n} $ has the $ p $-$(SR) $ property for each $ n\in\mathbb{N}, $ each $ P_{n}(K) $ is relatively weakly compact.\ It
follows from Lemma 3.4 \cite{ccl1} that $ K$ is relatively weakly compact.\
\end{proof}
\begin{prop}\label{p3}
Let $ (X_{n})_{n} $ be a sequence of Banach spaces.\ If $ 1 < r <\infty $ and $ 1 \leq p <\infty ,$
then each $ X_{n} $ has the $ p $-$ (SR^{\ast}) $ property if and only if $X=(\displaystyle\sum_{n=1}^{\infty}\oplus X_{n})_{\ell_{r}} $ has the same property.
\end{prop}
\begin{proof}
It is clear that if $X=(\displaystyle\sum_{n=1}^{\infty}\oplus X_{n})_{\ell_{r}} $ has the $ p $-$ (SR^{\ast})$ roperty, then each $ X_{n} $ has this property.\ Conversely,
let $ K$ be a $ p $-Right$ ^{\ast} $ subset of $ X. $\ It is clear that each $ \pi_{n}(K) $ is
also a $ p $-Right set.\ Since $ X_{n} $ has the $ p $-$(SR^{\ast}) $ property for each $ n\in\mathbb{N}, $ each $ \pi_{n}(K) $ is relatively weakly compact.\ It
follows from Lemma 3.4 \cite{ccl1} that $ K$ is relatively weakly compact.\
\end{proof}

Suppose that $ K $ is a bounded subset of Banach space $ X. $\ For $ 1\leq p\leq \infty, $
we set
\begin{center}
$\vartheta_{p}(K)=\inf\lbrace \hat{d}(A,K) : K\subset X^{\ast}$ is a $ p $-Right$ ^{\ast} $ set $ \rbrace. $
\end{center}
We can conclude that $ \vartheta_{p}(K)=0 $ if and only if $K\subset X$ is a $ p $-Right$ ^{\ast} $ set.\
For a bounded linear operator
$ T : X \rightarrow Y , $ we denote $ \vartheta_{p} (T(B_{X})) $ by $ \vartheta_{p} (T).$\\
The proof of the following theorem is similar to the proof of Theorem \ref{t4}, so its proof is omitted.\
\begin{thm}\label{t10}
Let $ X $ be a Banach space and $ 1\leq p < q\leq \infty. $\ The following statements are equivalent:\\
$\rm{(i)}$ For every Banach space $ Y, $ if for every Banach space $ Y, $ if $ T : Y \rightarrow X $ is an operator such that $ T^{\ast} $ is a Dunford-Pettis $ p $-convergent operator,
then $ T $ is a weakly $ q $-precompact (weakly $ q $-compact, $ q $-compact),\\
$\rm{(ii)}$ Same as $\rm{(i)}$ with $ Y=\ell_{\infty} ,$\\
$\rm{(iii)}$ Every $ p$-Right$ ^{\ast} $ subset of $ X^{\ast} $ is weakly $ q $-precompact (relatively weakly $q $-compact, $ q $-compact).
\end{thm}
%\begin{proof}
%We will show that in the relatively weakly$ q $-compact case.\ The other proof is similar.\\
%(i ) $ \Rightarrow $ (ii ) It is obvious.\\
%(ii ) $ \Rightarrow $ (iii ) Let $ K $ be a $ p $-Right$ ^{\ast} $ subset of $ X$ and let $ (x_{n}) $ be a sequence in $ K. $\ Define $ T:\ell_{1}\rightarrow X $
%by $ T(b)=\sum_{i} b_{i}x_{i}.$\ It is clear that that $ T^{\ast}(x^{\ast})=(x^{\ast}(x_{i}))_{i}. $\
%Let $ (x_{n}^{\ast})_{n} $ be a $ p $-Right null sequence in $ X^{\ast}. $\ Since $ K $ is a $ p $-Right$ ^{\ast} $ set, $$\lim_{n} \Vert T^{\ast}(x^{\ast}_{n}) \Vert =\lim_{n}\sup _{i}\vert x_{n}^{\ast}(x_{i}) \vert=0.$$So, $ T^{\ast}\in DPC_{p}(X^{\ast}, \ell_{\infty})$ and thus $ T$ is weakly $ q $-compact.\ Hence, $ (T(e_{n}^{1}))_{n}=(x^{\ast}_{n})_{n} $has a weakly $ q $-convergent subsequence.\\ (iii ) $ \Rightarrow $ (i ) Let $ T : Y \rightarrow X $ be an operator such that $ T^{\ast} $ is a Dunford-Pettis $ p $-convergent operator.\ Let $ (x^{\ast}_{n})_{n} $ be a $ p $-Right null sequence in $ X^{\ast}. $\ If $ y\in Y, $ then $ \vert x_{n}^{\ast}(T(y)) \vert\leq \Vert T^{\ast}(x^{\ast}_{n}) \Vert\rightarrow 0. $\ Therefore $ T(B_{Y}) $ is a $ p $-Right$ ^{\ast} $ subset of $ X. $\ Hence $ T(B_{Y}) $ is weakly $ q $-compact, and thus $ T$ is weakly $ q $-compact.\\ \end{proof}
\begin{cor}\label{c6}
Let $ X $ be a Banach space and $ 1\leq p < \infty. $\ The following statements are equivalent:\\
$\rm{(i)}$ For every Banach space $ Y, $ if $ T : Y \rightarrow X $ is an operator such that $ T^{\ast} $ is a Dunford-Pettis $ p $-convergent operator,
then $ T $ is weakly compact,\\
$\rm{(ii)}$ Same as $\rm{(i)}$ with $ Y=\ell_{1} ,$\\
$\rm{(iii)}$ $ X$ has the $ p $-$(SR^{\ast}) $ property,\\
$ \rm{(iv)} $ $ \omega(T^{\ast})\leq \vartheta_{p}(T^{\ast}) $ for every operator $ T$ from $ X$ into any Banach space $ Y, $\\
$ \rm{(v)} $ $ \omega(K) \leq \vartheta_{p}(K)$ for every bounded subset $ K$ of $ X. $
\end{cor}
%\begin{proof} (i ) $ \Rightarrow $ (ii ) It is obvious.\\ (ii ) $ \Rightarrow $ (iii ) Let $ K $ be a $ p $-Right$ ^{\ast} $ subset of $ X$ and let $ (x_{n}) $ be a sequence in $ K. $\ Define $ T:\ell_{1}\rightarrow X $ by $ T(b)=\sum_{i} b_{i}x_{i}.$\ It is clear that that $ T^{\ast}(x^{\ast})=(x^{\ast}(x_{i}))_{i}. $\ Let $ (x_{n}^{\ast})_{n} $ be a $ p $-Right null sequence in $ X^{\ast}. $\ Since $ K $ is a $ p $-Right$ ^{\ast} $ set, $$\lim_{n} \Vert T^{\ast}(x^{\ast}_{n}) \Vert =\lim_{n}\sup _{i}\vert x_{n}^{\ast}(x_{i}) \vert=0.$$ So, $ T^{\ast}\in DPC_{p}(X^{\ast}, \ell_{\infty})$ and thus $ T$ is weakly compact.\ Hence, $ (T(e_{n}^{1}))_{n}=(x^{\ast}_{n})_{n} $ has a weakly convergent subsequence.\\ (iii ) $ \Rightarrow $ (i ) Let $ T : Y \rightarrow X $ be an operator such that $ T^{\ast} $ is a Dunford-Pettis $ p $-convergent operator.\ Let $ (x^{\ast}_{n})_{n} $ be a $ p $-Right null sequence in $ X^{\ast}. $\ If $ y\in Y, $ then $ \vert x_{n}^{\ast}(T(y)) \vert\leq \Vert T^{\ast}(x^{\ast}_{n}) \Vert\rightarrow 0. $\ Therefore $ T(B_{Y}) $ is a $ p $-Right$ ^{\ast} $ subset of $ X. $\ Hence $ T(B_{Y}) $ is weakly compact, and thus $ T$ is weakly compact.\\ The equivalence of (iii) $ \Leftrightarrow$ (iv) and (iii) $ \Leftrightarrow $ (v) are straightforward. \end{proof}
\begin{cor}\label{c7}
If $ X ^{\ast} $ has the $ p $-$ (DPrcP)$ and $ Y$ has the $ p $-$ (SR^{\ast}) $ property, then $ L(X,Y) =W(X,Y).$\
\end{cor}
\begin{proof}
It can easily seen that continuous linear image each $ p $-Right null sequence is a $ p $-Right null sequence.\ Therefore, if $ T\in L(X,Y) $ and $ (y^{\ast}_{n})_{n} $ is a $ p $-Right null sequence in $ Y^{\ast} ,$ then $ (T^{\ast}(y^{\ast}_{n}))_{n} $ is a $ p $-Right null sequence in $ X^{\ast} .$\ Since $ X ^{\ast} $ has the $ p $-$ (DPrcP),$ $ \Vert T^{\ast}(y^{\ast}_{n}) \Vert\rightarrow 0$ and so, $ T^{\ast}\in DPC_{p}(Y^{\ast}, X^{\ast}). $\ Hence, Corollary \ref{c6} implies that $ T\in W(X, Y) .$\
\end{proof}

\section{$ (p,q) $-sequentially Right property on Banach spaces}
In this section, motivated by the class $\mathcal{P}_{p, q}$ in \cite{CAN} for those Banach spaces in which relatively $p$-compact sets are
relatively $q$-compact, we introduce the concepts of properties $(SR)_{p,q}$ and $(SR^{\ast})_{p,q}$ in order to find a
condition which every Dunford-Pettis $ q $-convergent operator is Dunford-Pettis $p$-convergent.

\begin{defn}\label{d2}
We say that $ X $ has the
$ (p,q) $-sequentially Right property (in short $ X $ has the $(SR)_{p,q}$ property), if each $ p$-Right set in $ X^{\ast} $ is a $ q $-Right set in $ X^{\ast} .$\
\end{defn}
\begin{defn}\label{d3} We say that $ X $ has the
$ (p,q) $-sequentially Right property (in short $ X $ has the $(SR^{\ast})_{p,q}$ property), if each $ p$-Right$ ^{\ast} $ set in $ X $ is a $q $-Right$^{\ast} $ set in $ X .$\
\end{defn}
From Definitions \ref{d2} and \ref{d3}, we have the following result. Since its proof is
obvious, the proof is omitted.
\begin{prop}\label{p4} If $ X^\ast $ has the $ (SR)_{p,q} $ property,
then $ X $ has the $ (SR^{\ast})_{p,q} $ property.\
\end{prop}
\begin{thm}\label{t11}
Let $ 1\leq p < q \leq \infty.$\ The following statements are equivalent:
\begin{enumerate}
\item[(i)] $X$ has the $ (SR)_{p,q}$ property.\
\item [(ii)] $DPC_p(X, Y)\subseteq DPC_q(X, Y)$, for every Banach space $ Y.$\
\item [(iii)] Same as $ \rm{(ii)} $ for $ Y=\ell_{\infty}.$
\end{enumerate}
\end{thm}
\begin{proof}
(i) $ \Rightarrow $ (ii) If $ T\in DPC_{p}(X,Y),$ then the part $ \rm{(i)} $ of Lemma \ref{l1} implies that $ T^{\ast}(B_{Y^{\ast}}) $ is a $ p $-Right set.\ Since $X$ has the $ (SR)_{p,q}$ property, $ T^{\ast}(B_{Y^{\ast}}) $ is a $ q $-Right set.\
Therefore, the part $ \rm{(i)} $ of Lemma \ref{l1} yields that $ T\in DPC_{q}(X,Y).$\\
(ii) $ \Rightarrow $ (iii) is obvious.\\
(iii) $ \Rightarrow $ (i)
Suppose that $ K$ is a $ p$-Right set in $X^{\ast} $ and let $ (x_{n}^{\ast})_{n} $
is an arbitrary sequence in $ K. $\
Assume that $ T:\ell_{1}\rightarrow X^{\ast}$
is defined by $ T(b_{n})=\sum_{n=1}^{\infty} b_{n}x_{n}^{\ast} .$\
It is clear that $ T^{\ast}(x) =( x_{i}^{\ast}(x ))_{i}$ for all $x\in X.$\ Suppose that the sequence $ (x_{n})_{n} $ is a $ p $-Right null sequence in $ X. $\
Since, $ K $ is a $ p$-Right set, we have
$$ \displaystyle \lim_{n}\sup_{i}\vert x_{i}^{\ast}(x_{n}) \vert =0.$$\
Therefore, $ \displaystyle \lim_{n}\Vert T^{\ast}(x_{n}) \Vert= 0.$\ Hence, $ T^{\ast}_{\vert_{X}} $
is a Dunford-Pettis $ p$-convergent operator and so by the assumption $ T^{\ast}_{\vert_{X}} $
is a Dunford-Pettis $ q$-convergent operator.\ Now, assume that $ (x_{n})_{n} $ is a $ q$-Right null sequence in $ X$ and $y\in B_{\ell_{1}}.
$\ Hence,
$$ \vert T(y)(x_{n}) \vert = \vert T^{\ast}(x_{n})(y) \vert\leq \Vert T^{\ast}(x_{n}) \Vert\rightarrow 0 .$$\ Therefore, $ T(B_{\ell_{1}}) $ is a $ q$-Right set in $ X^{\ast}$ which
follows that $ (x_{n}^{\ast})_{n} $ is also a $ q$-Right set in $ X^{\ast}.$\ Since $ (x_{n}^{\ast})_{n} $ is an arbitrary sequence in $K,$ $K$ is a $ q$-Right set.\ Thus, $ X$ has the $(SR)_{p,q}$ property.\
\end{proof}
\begin{cor}\label{c8}
If every $ p $-Right set in $ X^{\ast} $ is relatively compact, then $ X$ has the $ (SR)_{p,q}$ property.\
\end{cor}
\begin{cor}\label{c9}
Let $ 1\leq p < q \leq \infty.$\ The following statements hold.
\begin{enumerate}
\item[(i)] If $ X$ has both properties $ (SR)_{p,q}$ and $ p $-$ (DPrcP) $, then $X$ has the $ q $-$ (DPrcP). $\
\item [(ii)] If $ X^{\ast\ast} $ has both properties $ (SR)_{p,q}$ and $ p $-$ (DPrcP) $, then $X$ has the $ q $-$ (DPrcP). $\
\item[(iii)] If $ X$ has the $ p $-$ (SR ) ,$ then $ X$ has the $ (SR)_{p,q}$ property.\
\end{enumerate}
\end{cor}
\begin{proof}
$ \rm{(i)} $ Suppose that $ T : X \rightarrow Y $ is a bounded linear operator.\ Since $X$ has the $ p $-$ (DPrcP) $, then $T\in DPC_p(X, Y).$\ On the other hand, $ X$ has property $ (SR)_{p,q},$ thus by Theorem \ref{t11}, $ T\in DPC_q(X, Y)$.\ Thus, $X$ has the $ q $-$ (DPrcP) .$\\
$ \rm{(ii)} $ By part (i), $ X^{\ast\ast}$ has the $ q $-$ (DPrcP).$\ Hence, $ X$ has the $ q $-$ (DPrcP) .$\\
$ \rm{(iii)} $ Let $Y$ be a Banach space and $ T \in DPC_p (X, Y).$\ From part $ \rm{(i)} $ of Lemma \ref{l1}, $T^{\ast}(B_{Y^{\ast}})$ is a $p$-Right set.\ Since $ X $ has the $ {p}$-$(SR) $ property, $T^{\ast}(B_{Y^{\ast}})$ is relatively weakly compact.\ It is clear that $T$ is weakly compact and so, $ T$ is Dunford-Pettis completely continuous.\ Thus, $ T $ is Dunford-Pettis $ q $-convergent.\ Hence,
by Theorem \ref{t11} $ X $ has the $ (SR)_{p,q}$ property.
\end{proof}
\begin{thm}\label{t12}
Let $ 1 <p \leq\infty$. The following statements are equivalent.
\begin{enumerate}
\item[(i)] $ X$ has the $ p $-$ (DPrcP). $
\item[(ii)] $ X $ has the $ (SR)_{1,p} $ property and
$ X$ contains no isomorphic copy of $ c_{0}.$
\end{enumerate}
\end{thm}
\begin{proof}
(i) $ \Rightarrow $ (ii) is obvious.\\
(ii) $ \Rightarrow $ (i) Since $ X $ contains no isomorphic copy of $ c_{0} ,$ $X$ has the $ 1 $-Schur property (see, Theorem 2.4 in \cite{dm}) and so has the $ 1 $-$ (DPrcP). $\ Hence,
$B_{X^{\ast}} $ is
$1$-Right subset of $ X^{\ast}. $\ Since
$ X$ has the $ (SR)_{1,p} $ property, $ B_{X^{\ast}} $ is a $ p $-Right set.\ It is easy to verify that
$ X $ has the $ p $-$ (DPrcP). $
\end{proof}

In the sequel, we characterize property $ (SR^{\ast})_{p,q}.$\ Since
the proof of the following theorem is similar to the proof of Theorem \ref{t11}, its proof is omitted.\
\begin{thm}\label{t13}Let $ 1\leq p <q\leq \infty.$\
The following statements are equivalent.\
\begin{enumerate}
\item[(i)] $ X $ has the $ (SR^{\ast})_{p,q}$ property.
\item[(ii)] $ DPC_{p} (X^{\ast},Y^{\ast})\subseteq DPC_{q}(X^{\ast},Y^{\ast}),$
for every Banach space $ Y . $
\item[(iii)] Same as $\rm{ (ii)} $ for $ Y=\ell_{1}.$
\end{enumerate}
\end{thm}

\begin{cor} \label{c10}
If $ X^{\ast}$ has the $ p $-$ (SR) $ property,
then
$ X$ has the $ (SR^{\ast})_{p,q}$ property.
\end{cor}
\begin{proof}
Let $ Y $ be an arbitrary Banach space and $ T\in L(Y,X) $ such that $ T^{\ast} $ be a Dunford-Pettis $ p $-convergent operator.\ Therefore, the part (iii) of
Lemma \ref{l1}, $ T(B_{Y}) $ is a $ p $-Right$ ^{\ast} $ set in $ X. $\ Since $ X^{\ast} $ has the $ p $-$ (SR) $ property, the part (iii) of Proposition \ref{p2} implies that
$ X $ has the $ p $-$ (SR^{\ast})$ property.\ Hence, $ T(B_{Y}) $ is a relatively weakly compact set in $ X $ and so $ T $ is weakly compact.\ Thus, $ T^{\ast} $ is weakly compact and so $ T^{\ast} $ is Dunford-Pettis  completely continuous.\ So, $ T^{\ast} $ is a
Dunford-Pettis $ q $-convergent operator.\ Hence, as an immediate
consequence of the Theorem \ref{t13}, we can conclude that
$ X $ has the $ (SR^{\ast})_{p,q}$ property.
\end{proof}

\begin{thm}\label{t14} If $ 1 <p\leq \infty$, then the following statements are equivalent.
\begin{enumerate}
\item[(i)] $ X^{\ast}$ has the $ p $-$ (DPrcP). $\
\item[(ii)] $ X $ has the $ (SR^{\ast})_{1,p} $ property and $ X^{\ast}$ contains no isomorphic copy of $ c_{0}.$\
\end{enumerate}
\end{thm}
\begin{proof}
(i) $ \Rightarrow $ (ii) Suppose that $ X^{\ast}$ has the $ p $-$ (DPrcP). $\ By Theorem \ref{t14}, $ X^{\ast} $ has the
$ (SR)_{1,p}$ property and $ X^{\ast}$ contains no isomorphic copy of $ c_{0}.$\ Thus, Proposition \ref{p4} implies that $ X$ has the $(SR^{\ast})_{1,p} $ property.\\
(ii) $ \Rightarrow $ (i) By the hypothesis $X^{\ast}$ contains no isomorphic copy of $ c_{0}.$\ Therefore Theorem 2.4 in \cite{dm} implies that $X^\ast$ has the $ 1 $-Schur property and so $X^\ast$ has the $ 1 $-$ (DPrcP). $\
Therefore, by the part (iii) of Lemma \ref{l1}, $ B_{X} $ is a $ {1} $-Right$ ^{\ast} $ set in $ X.$\ Since $X$ has the $ ( SR^{\ast})_{1,p}$ property, $ B_{X} $ is a $ p $-Right$ ^{\ast} $
set.\ Hence, by reapplying the part (iii) of Lemma \ref{l1}, $ X^{\ast}$ has the $ p $-$ (DPrcP). $\
\end{proof}
Finally,  we present an example of property $ (SR)_{p,q} $ and an example of property $(SR^{\ast})_{p,q}.$\
\begin{ex}\rm\label{e1}
$ \rm{(i)} $ If $ \Omega $ is a compact Hausdorff space, then
$C(\Omega)$ has the $ (SR)_{p,q} $ property.\\
$ \rm{(ii)} $ If $
(\Omega,\Sigma,\mu) $ is any $ \sigma $-finite measure space, then $ L_{1}(\mu) $ has the $(SR^{\ast})_{p,q}$ property.
\end{ex}

\end{document}